\documentclass[12pt]{article}

\usepackage[top=0.75in]{geometry}
\usepackage{amsmath,amsthm,amssymb}
\usepackage{enumerate}

\newtheorem{thm}{Theorem}
\newtheorem{lem}[thm]{Lemma}
\newtheorem{prop}[thm]{Proposition}
\newtheorem{cor}[thm]{Corollary}

\theoremstyle{definition}

\def\lf{\left\lfloor}   
\def\rf{\right\rfloor}

\begin{document}

\title{A Probabilistic Two-Pile Game}
\author{Ho-Hon Leung, Thotsaporn ``Aek'' Thanatipanonda }
\date{March 7, 2019}

\maketitle
\thispagestyle{empty}

\begin{abstract}
We consider a game with two piles, in which two players take turn to 
add $a$ or $b$ chips ($a$, $b$ are not necessarily positive) 
randomly and independently to their respective piles. The player
who collects $n$ chips first wins the game. We derive general
formulas for $p_n$, the probability of the second player winning the game by collecting $n$ chips first and
show the calculation for the cases $\{a,b\}$ = $\{-1,1\}$ and $\{-1,2\}$.
The latter case was asked by Wong and Xu \cite{WX}. At the end,
we derive the general formula for $p_{n_1,n_2}$, the 
probability of the second player winning the game by collecting $n_2$ chips before the 
first player collects $n_1$ chips.   \\
\end{abstract}

\noindent \textbf{Accompanied Maple program} \\
This paper comes with the Maple program \texttt{Piles}
which can be found on the second author's website 
\texttt{www.thotsaporn.com}  .

\section{Introduction} \label{section1}
Game is the source of motivation to do mathematics.
A will to win a game is the motivation to figure out mathematics behind it. It is evident that pile games have been played since ancient times. For examples, Nim, Wythoff and their variants are some of the classical pile games. In such games, two players take turn to remove chips from the existing pile(s). 
The rule of removing chips from the pile(s) varies according to the game. 
In Nim, a game with multiple piles, each player may remove any number of chips from one of the available piles. The player who takes the last chip loses the game (misere game). In some other games, the player who takes the 
last chip wins the game (normal play). 
In Wythoff, a game with two piles, a player is allowed to remove any number of chips from one or both piles; when removing chips from both piles, the number of chips removed 
from each pile must be equal. The player who takes the last chip (remove any chip) 
wins the game. These games are quite well known and well studied. For example, 
the ``bible of combinatorial game theory'', {\it Winning ways for your Mathematical Plays}, written 
by Berlekamp, Conway and Guy \cite{WW} is a good reference for a mathematical introduction to these games.

In this paper, we consider a more general version of the game investigated by Wong and Xu \cite{WX}, where two players
take turn to collect a specific number of chips randomly and independently in order to build their own piles (instead of `remove any chip from their own piles'). The player who collects $n$ chips first is the winner. 
Formally speaking, let $n$ be a fixed non-negative integer. In our game, both players start without any chip. 
In each turn, each player flips a fair coin to decide whether to
add $a$ or $b$ chips ($a$, $b$ are {\it not necessarily positive}) to their own piles. The player who collects $n$ chips first is the winner.

Let the random variable $S_k$ be the number of chips collected by a player on his $k^{th}$ move. Let $A$ be the first player and $B$ be the second player. The chance that $B$ wins the game by collecting $n$ chips first is
\begin{align} p_n = \sum_{k=1}^{\infty}  \label{equation0}
P( \mbox{A does not win on his } k^{th} \mbox{ move}) \cdot
P( \mbox{B wins on his } k^{th} \mbox{ move}).  \end{align}

\noindent \textbf{Important notation} 
\begin{itemize}
\item $p_n$ = the  probability for the second player to win the game by collecting $n$ chips first. 
\item $q(n,k)$ = the probability that a player does not win the game on his $k^{th}$ move, i.e., he never collects $n$ chips on or before his $k^{th}$ move.
\item $r(n,k)$ = the probability that a player collects $n$ chips for the first time on his $k^{th}$ move.
\end{itemize}

\noindent The equation (\ref{equation0}) can be written as follows:
\begin{align}  \label{M1}
 p_n = \sum_{k=1}^{\infty} q(n,k) \cdot r(n,k)
\end{align}
where $q(n,k) = P(S_j < n\mbox{ for all }j=0,1,\dots,k)$ 
and $r(n,k) = P(S_k \geq n \mbox{ and } S_{j} < n \mbox{ for all }j=0,1,\dots, k-1) $.\\

\noindent There is a nice connection between the probabilities $q(n,k)$ and $r(n,k)$ as follows: \\

\noindent \textbf{Claim 1:} For any positive integer $n$, 
\[ r(n,k) = q(n,k-1)-q(n,k), \;\ \;\ k=1,2,\dots. \] 
This can be shown by
\begin{align*}
r(n,k) &= P(S_k \geq n \mbox{ and } S_{j} < n,\mbox{ for all }j=0,1,\dots,k-1) \\
&= P(S_k \geq n) - P(S_{k-1} \geq n) \\
&= (1-q(n,k))-(1-q(n,k-1))  \\ 
&= q(n,k-1)-q(n,k).
\end{align*}

\noindent \textbf{Remark:} By Claim 1 and the fact that
$q(n,0)=1$ for all $n \geq 1$,
we write $q(n,k)$ where $ k \geq 1$, by
\begin{equation}  \label{M2}   
q(n,k) = 1- \sum_{j=1}^k r(n,j).            
\end{equation} 

\noindent The following claim works under the condition that the 
game will not go on indefinitely (the probability that either 
one of the players wins the game is 1). \\ 

\noindent \textbf{Claim 2:}
For any fixed positive integer $n$, 
if $\displaystyle \lim_{k \to \infty}q(n,k) =0$  
(i.e., $a+b \geq 0$), then
 \[\displaystyle \sum_{k=1}^{\infty} [q(n,k-1)+q(n,k)] \cdot r(n,k) = 1.\]
This can be shown by
\begin{align*}
& \sum_{k=1}^{\infty} [q(n,k-1)+q(n,k)] \cdot r(n,k) \\
&=  \sum_{k=1}^{\infty} [q(n,k-1)+q(n,k)] \cdot [q(n,k-1)-q(n,k)] 
\;\ \;\ \mbox{ from Claim 1} \\
&= \sum_{k=1}^{\infty} [q^2(n,k-1)-q^2(n,k)] \\
&= q^2(n,0)- \lim_{k \to \infty}q^2(n,k)  \;\ \;\ =1. 
\end{align*}

\noindent Claim 2 can also be explained combinatorially as 
\begin{align*}
1 &= \mbox{$P$( first player wins) + $P$( second player wins) + $P$( nobody wins) }\\
 &= \sum_{k=1}^{\infty} r(n,k)q(n,k-1) + \sum_{k=1}^{\infty} q(n,k)r(n,k)
 + \lim_{k \to \infty}q^2(n,k).
\end{align*}

In Theorem \ref{theorem1} below, we write the probability $p_n$ in terms of the probabilities $r(n,k)$ alone.
\begin{thm}\label{theorem1}
If $\displaystyle \lim_{k \to \infty}q(n,k)=0$, then
$\displaystyle p_n = \dfrac{1}{2}-\dfrac{1}{2}\sum_{k=1}^{\infty} r^2(n,k)$.
\end{thm}

\begin{proof}
By definition,   \[\displaystyle p_n = \sum_{k=1}^{\infty} q(n,k) \cdot r(n,k).  \]
On the other hand, based on Claim 2, we have  \[\displaystyle 
p_n = 1-\sum_{k=1}^{\infty} q(n,k-1) \cdot r(n,k).  \]
Then, by combining the first two equations above, we obtain
\[  p_n = \dfrac{1}{2}-\dfrac{1}{2}\sum_{k=1}^{\infty} [q(n,k-1)-q(n,k)] \cdot r(n,k). \]
Then, the result follows by applying Claim 1.
\end{proof}

\noindent 
In order to find $p_n$, by Theorem \ref{theorem1}, we just need 
to compute the probabilities $r(n,k)$ and find the sum of squares of them. 
For example, for the case $\{a,b\}=\{1,2\}$, the probability
$r(n,k)$ is given by the following expression \cite[Theorem 3]{WX}:
\[  r(n,k) = \dfrac{1}{2^k}\left[ \binom{k}{n-k}+\binom{k-1}{n-k} \right].\]
Although it can be shown that the expression $\displaystyle \sum_{k=1}^{\infty} r^2(n,k)$
does not have a closed-form formula 
(strictly speaking, the sum is not Gosper-summable. 
See Chapter 8 of the book `$A=B$' \cite{AB} for reference),
its approximation was done by Wong and Xu \cite[p.12]{WX}. In particular, they showed that   
\[  \displaystyle \sum_{k=1}^{\infty} r^2(n,k)\sim\sqrt{\dfrac{27}{8\pi n}}\]when $n$ is large.

\noindent Furthermore, they worked on the cases $a>0, b>0$
and asked the reader to investigate the case $\{a,b\}=\{-1,2\}.$
We will answer their question along with the answer 
to the case $\{a,b\}=\{-1,1\}$.

\section{The case $\{a,b\}=\{-1,1\}$} \label{section2}

Each player adds/removes one chip with 
probability 1/2 to/from his pile. 
The pile is allowed to have a negative number of chips. 
The first player who collects $n$ chips wins the game.  \\ 

\noindent If $n=0$, the first player will always win as both players start without any chip. 
The probability for the second player to win the game is $0$, i.e., $p_0 = 0$.

\subsection{The winning probability for the first non-trivial case: $n=1$} \label{section2.1}

In this subsection, we simplify our notation slightly. The probabilities $q(1,k)$ and $r(1,k)$ defined in Section \ref{section1} are abbreviated to $q(k)$ and $r(k)$ respectively.\\

\noindent Let $C(k)$ be the number of ways for a player to have no chip on his $k^{th}$ move without ever collecting one chip 
(so the game still goes on). \\ \\
At this point, the reader may notice that the number $C(k)$ is
related to the famous Catalan numbers. 
In fact, for $m=1,2,\dots$, we have
\[C(2m-1) = 0\]and
\[C(2m)= \dfrac{\binom{2m}{m}}{m+1}.\]

\noindent The probability that the second player collects one chip for the first time on his $k^{th}$ move is
\[r(k) = \dfrac{C(k-1)}{2^k}\]because the player has
no chip on his $(k-1)^{th}$ move, and on his next move he must collect one chip to win the game.
Hence, for $m=1,2,\dots$, we have
\[r(2m)=0\]and
\[r(2m-1) = \dfrac{(2m-2)!}{m!(m-1)!}\cdot \dfrac{2}{4^m}.\]

\noindent We would like to apply Theorem \ref{theorem1} to find $p_1$,
the probability of the second player winning the game by
collecting one chip first. First, we need to verify that the condition  
$\displaystyle \lim_{k \to \infty}q( k)=0$ is true, i.e., 
the probability that one of the players wins the game by collecting one chip is 1. 
Equivalently, by \eqref{M2}, we show that 
\begin{align} \label{equation100} \sum_{k=1}^{\infty} r(k) &= 1.  \end{align}

\noindent Fortunately, there are many ways to evaluate this sum.  
The second author's favorite tool for evaluating geometric sums
(the binomial sum is a geometric sum) is Gosper's Algorithm.
The detail of the algorithm was beautifully explained in Chapter 5 of the book `$A=B$', a masterpiece written by Petkovsek, Wilf and Zeilberger \cite{AB}. 
This algorithm has been implemented 
in all major symbolic computation programs
like Maple and Mathematica. For example, in Maple, one conveniently types \\ \\
 \texttt{ sum(  (2*m-2)!/m!/(m-1)!*2/4\string^m, m=1..M);} \\ \\
then Maple will return the expression $1-\dfrac{(2M)!}{M!M!4^M}$. By an application of the Stirling's
formula: $n! \approx \sqrt{2\pi n} \left(\dfrac{n}{e}\right)^n$, we note that \[1-\dfrac{(2M)!}{M!M!4^M} \rightarrow 1 \mbox{ as }M\rightarrow \infty.  \]Hence, the equation (\ref{equation100}) is true.

\noindent We are ready to find $p_1$ by Theorem \ref{theorem1}. To evaluate 
\[ \sum_{k=1}^{\infty} r^2(k),  \]
in Maple, we type \\ \\
\texttt{ sum( ( (2*m-2)!/m!/(m-1)!*2/4\string^m )\string^2, m=1..infinity);} \\ \\
and then Maple returns $\dfrac{4}{\pi}-1$. By Theorem \ref{theorem1}, 
\[  p_1 = \dfrac{1}{2}-\dfrac{1}{2}\left(\dfrac{4}{\pi}-1\right) 
\approx 0.3633802277. \]

\noindent Hence, the value of $p_1$ is conveniently
obtained by applying Theorem \ref{theorem1} and the summation tools in Maple. 

\subsection{ The winning probabilities for the cases $n\geq 2$} \label{section2.2}

Analogously, let $C(n,k)$ be the number of ways for a player to 
have $n-1$ chips on his $k^{th}$ move without ever collecting $n$ chips.

\begin{lem} \label{lemma2}
For $n \geq 0, \;\ k \geq 1$ such that $(n-k) \equiv 0 \mod 2$, then \[C(n,k) = 0.\]
Otherwise, for $s \geq 0,$
\begin{align*} 
C(2s+1,2m) &= \dfrac{2s+1}{m+s+1}\binom{2m}{m-s}
 , \;\ \;\ \mbox{for } m=0,1,2,\dots, \\
C(2s,2m-1) &= \dfrac{s}{m}\binom{2m}{m-s}
, \;\ \;\ \mbox{for } m=1,2,\dots,.
\end{align*}
\end{lem}

\begin{proof}
The main recurrence relation for the numbers $C(n,k)$ is
\[   C(n,k) = C(n-1,k-1) +C(n+1,k-1), \;\  n \geq 1, \;\ k \geq 1.  \]
If the first move is 1, then the number of ways for a player to have $n-1$ chips on his $k^{th}$ move is $C(n-1,k-1)$. 
On the other hand, if the first move is $-1$, then this number is $C(n+1,k-1)$. We rearrange terms and shift variables to obtain
\begin{equation} \label{M3}
  C(n,k) = C(n-1,k+1) - C(n-2,k), \;\  n \geq 2, \;\ k \geq 0.  
 \end{equation}
 Then the results follow by induction on $n$ 
 where base cases are given by \[C(0,k) =0, \quad C(1,2m-1)=0 \mbox{ and } 
 C(1,2m)= \dfrac{\binom{2m}{m}}{m+1}.\]
 It is also important to note that $C(n,0)=0$ for all $n \geq 0$ except $C(1,0) =1.$ 

\end{proof}

\noindent We proceed the same way as in Section \ref{section2.1}. The following relation between the numbers $r(n,k)$ and $C(n,k-1)$ is true for $n\geq 0$: \begin{align} \label{equation101} r(n,k)&= \dfrac{C(n,k-1)}{2^k}.\end{align}By Lemma \ref{lemma2}, for $n \geq 0, \;\ k \geq 1$ and $n-k \equiv 1 \mod 2$, we have 
\begin{equation} \label{zero} 
r(n,k) = 0.
\end{equation} 
Otherwise, for $s \geq 0, $
\begin{align*} 
r(2s+1,2m+1) = \dfrac{C(2s+1,2m)}{2^{2m+1}} 
&= \dfrac{2s+1}{(m+s+1)\cdot 2\cdot 4^m}\binom{2m}{m-s},
 \;\ \;\ \mbox{for } m=0,1,2,\dots, \\
r(2s,2m) = \dfrac{C(2s,2m-1)}{2^{2m}} &= \dfrac{s}{m\cdot 4^m}\binom{2m}{m-s}
, \;\ \;\ \mbox{for } m=1,2,\dots,.
\end{align*}

\noindent To apply Theorem \ref{theorem1}, it is necessary to prove the following lemma.

\begin{lem} \label{lemma3}
For each $n \geq 1$, the probability that one of the players wins the game (by collecting $n$ chips first)
is 1, i.e.,
\[ \sum_{k=1}^{\infty}r(n,k) = 1. \]
\end{lem}

\begin{proof}
This can be done by induction on $n$. By (\ref{M3}) and (\ref{equation101}), we have the following recurrence relation for the numbers $r(n,k)$,
\begin{align} \nonumber r(n,k) &= \dfrac{C(n,k-1)}{2^k} = \dfrac{C(n-1,k)}{2^k}-\dfrac{C(n-2,k-1)}{2^k} \\
&= 2r(n-1,k+1) - r(n-2,k). \label{equation102} \end{align}
By (\ref{equation100}), we note that \begin{align} \label{equation103} \displaystyle \sum_{k=1}^{\infty}r(1,k) &= 1.\end{align}
For $n=2$, by (\ref{equation102}) and (\ref{equation103}), we get
\begin{align*}  \sum_{k=1}^{\infty}r(2,k) &=  \sum_{k=1}^{\infty} (2r(1,k+1)-r(0,k))
= 2\sum_{k=1}^{\infty} r(1,k+1) +2r(1,1) -2r(1,1)  \\
&=  2\sum_{k=1}^{\infty} r(1,k) -2r(1,1) = 2-1 =1,
\end{align*}
since $r(0,k) =0$ for all $k \geq 1.$ 

\noindent For the cases $n \geq 3$, by (\ref{equation102}) and the inductive hypothesis, we have
\begin{align*}  \sum_{k=1}^{\infty}r(n,k) &=  \sum_{k=1}^{\infty} (2r(n-1,k+1)-r(n-2,k))\\
&= 2\sum_{k=1}^{\infty} r(n-1,k) -\sum_{k=1}^{\infty}r(n-2,k) = 2-1 =1,
\end{align*}
since $r(n-1,1) =0$ for all $n \geq 3.$
\end{proof}

\noindent Finally, we are in the position to apply Theorem \ref{theorem1} for each value of $n$.
For example, to evaluate 
$\displaystyle \sum_{k=1}^{\infty} r^2(2,k)$, we type \\

\noindent \texttt{ sum( (binomial(2*m,m-1)/m/4\string^m)\string^2  ,m=1..infinity);} \\

\noindent and then Maple returns $\dfrac{16}{\pi}-5$. By Theorem \ref{theorem1},
\[  p_2 = \dfrac{1}{2}-\dfrac{1}{2}(\dfrac{16}{\pi}-5) \approx 0.4535209109.  \]

\noindent The other values of $p_n$ are  \\ \\
$ n= 3, \;\ \displaystyle \sum_{k=1}^{\infty} r^2(3,k) = \dfrac{236}{3\pi}-25$ 
and $p_3 \approx 0.4798111434 $,    \\
$ n= 4, \;\ \displaystyle \sum_{k=1}^{\infty} r^2(4,k) = \dfrac{1216}{3\pi}-129$ 
and $p_4 \approx 0.4891964033 $,    \\
$ n= 5, \;\ \displaystyle \sum_{k=1}^{\infty} r^2(5,k) = \dfrac{32092}{15\pi}-681$ 
and $p_5 \approx 0.4933044576 $,    \\
$ n= 6, \;\ \displaystyle \sum_{k=1}^{\infty} r^2(6,k) = \dfrac{172144}{15\pi}-3653$ 
and $p_6 \approx 0.4954322531. $    \\
$\dots.$ \\ \\
There is an interesting pattern for the values of $\displaystyle \sum_{k=1}^{\infty} r^2(n,k)$ for $n=1,2,\dots$. 
In fact, let $T_n$ be $\displaystyle \sum_{k=1}^{\infty} r^2(n,k)$. The terms 
$T_n$ satisfy a recurrence relation with polynomial coefficients:
\[    (n+3)T_{n+3}-(7n+16)T_{n+2}+(7n+5)T_{n+1}-n T_n=0.     \]
(\textbf{Remark:} The guessing recurrence relation was found by a holonomic ansatz, i.e., we assume that 
the sequence $T_n$ satisfies a relation of the form
\[ (a_0n+b_0)T_{n}+(a_1n+b_1)T_{n+1}+\dots+(a_Nn+b_N)T_{n+N}=0\]
for some specific $N$. Then, by plugging in values of $n, T_n, T_{n+1}, \dots, T_{n+N}$, 
say, for $n=1,2,\dots,30$, we can solve the system of linear equations for the unknowns $a_i$ and $b_i$.) 
The authors believe that it can be shown formally by Zeilberger's algorithm, 
see Chapter 6 of the book `A=B' \cite{AB} and Zeilberger's Maple package \texttt{Ekhad}. 
However, since the terms $T_n$ are separated
into two cases ($n$ is odd or $n$ is even), the proof may not be straightforward.

\subsection{The winning probability within $k$ moves} \label{section2.3}
The case $\{a,b\} = \{-1,1\}$  is so nice that
we can even answer more questions than the one that was asked in the first place.
In this subsection, for a fixed positive integer $n$, we find the probability that the second player wins within $k$ moves.
All this could be done, thanks once again to Gosper's algorithm!  \\

\noindent First, we demonstrate it for the case $n=1$. The other
cases can be done in the same way, although we have to do it
case by case. We recall from Section \ref{section2.1} that 
\[r(2m) =0  \]
and 
\[ r(2m-1) = \dfrac{(2m-2)!}{m!(m-1)!}\cdot \dfrac{2}{4^m} .   \]

\noindent By the remark after Claim 1, 
\[   q(2m)=q(2m-1) = 1-\sum_{j=1}^m r(2j-1) 
=  1-\sum_{j=1}^m \dfrac{(2j-2)!}{j!(j-1)!}\cdot \dfrac{2}{4^j}.\]
This finite sum is Gosper-summable. In Maple, we type \\

\noindent \texttt{ simplify( 1-sum( (2*j-2)!*2/j!/(j-1)!/4\string^j  ,j=1..m));} \\ \\
and then the output is $\dfrac{\binom{2m}{m}}{4^m}.$ 
(The expression is so nice! It is screaming for a combinatorial explanation!) \\

\noindent Lastly, the probability that the second player 
wins within $k$ moves (the partial sum of \eqref{M1}) is

\begin{align*} 
\sum_{i=1}^k q(i)r(i)  &= \sum_{j=1}^{\lf\frac{k+1}{2}\rf} q(2j-1)r(2j-1)\\
&= \sum_{j=1}^{\lf\frac{k+1}{2}\rf} \dfrac{\binom{2j}{j}}{4^j}\dfrac{(2j-2)!}{j!(j-1)!}\cdot \dfrac{2}{4^j}\\
&= \sum_{j=1}^{\lf\frac{k+1}{2}\rf} \binom{2j}{j}^2 \cdot \dfrac{1}{16^j}\cdot \dfrac{1}{2j-1}.\\
\end{align*}

\noindent In Maple, we type\\

\noindent \texttt{ sum(binomial(2*j,j)\string^2/16\string^j/(2*j-1),j=1..L); } \\ \\
and then we again have a closed-form solution. The probability that the second player wins within $k$ moves is \\
\[  1-\dfrac{2L+1}{16^L}\binom{2L}{L}^2, 
\;\ \;\ \mbox{where } L = \lf\frac{k+1}{2}\rf.          \]

\subsection{Average duration of a game play} \label{section2.4}
The question about the average duration of a game play 
comes up naturally from the work in the previous
section. It would also be interesting to do research on average duration of different
kind of games. As a reference, we refer to the work done earlier by Robinson and Vijay on the duration 
of the game \textit{Dreidel} \cite{RS}. In this subsection,
we let a random variable $X$ be the number of moves required to end the game.

\begin{thm}
For any take-away move $\{a,b\}$ where $a$ and $b$ are not necessarily positive,
the average duration of the game is 
$\displaystyle \sum_{k=0}^{\infty} q(n,k)^2.$
\end{thm}

\begin{proof}We compute $E[X]$ as follows:
\begin{align*}
  E[X] &= \sum_{k=1}^{\infty} k \cdot P(\mbox{ the game ends at the } k^{th} \mbox{ turn})  \\ 
&= \sum_{k=1}^{\infty} k \cdot  [q(n,k-1)r(n,k)+q(n,k)r(n,k)] \\
&= \sum_{k=1}^{\infty} k \cdot  [q(n,k-1)^2-q(n,k)^2] \\
&= \sum_{k=1}^{\infty} q(n,k-1)^2.
\end{align*}
The result follows by shifting the index $k$ by 1.
\end{proof}

\begin{cor}
For the case $\{a,b\}=\{-1,1\},$  $E[X]= \infty$ for any $n \geq 1$.
\end{cor}

\begin{proof}
For $n =1$, by our computation of the probability $q(k)$ in Section \ref{section2.3}, we have \[ E[X] = \sum_{k=0}^{\infty} q(k)^2= 1+2 \cdot \sum_{m=1}^\infty
 \left[\dfrac{\binom{2m}{m}}{4^m} \right]^2 = \infty. \]
The last equality was found by Maple. 
But it can also be seen through the following approximation:  \[\left[\dfrac{\binom{2m}{m}}{4^m} \right]^2 
\approx \dfrac{1}{\pi m} \mbox{ as } m \to \infty.\] 
The sum indeed diverges but the rate of divergence is as slow as the harmonic series.
For $n > 1,$ it is clear that the average duration of the game
will be longer than $n=1$. Therefore, $E[X]= \infty$ too for $n>1$.
\end{proof}

\section{The case $\{a,b\}=\{-1,2\}$} \label{section3}

For this case, each player is allowed to add two chips to his pile or remove one chip from his pile, each with probability $1/2$. The pile is allowed
to have a negative number of chips. 
The first player who collects $n$ chips wins the game.  \\ \\
Let $D(n,k)$ be the number of ways for a player to have $n-1$ or $n-2$ chips on his $k^{th}$ move without ever collecting $n$ chips 
(so the game still goes on). 

\begin{lem}
The numbers $D(n,k)$ satisfy the following recurrence relation,
\[  D(n,k) = D(n-1,k+1)-D(n-3,k)  \] with
base cases $D(-1,k)=D(0,k) = 0$ and 
\[D(1,3m)=\dfrac{\binom{3m}{m}}{2m+1}, \quad
D(1,3m+1)=\dfrac{\binom{3m+1}{m+1}}{2m+1}
, \quad D(1,3m+2)=0.\]
\end{lem}

\begin{proof}
The recurrence relation arises from whether the first move is $+2$ or $-1$,
\[D(n,k) = D(n-2,k-1)+ D(n+1,k-1).\]
After shifting indexes and rearranging terms, we get the desired recurrence relation.\\

\noindent For the base cases, the sequence of non-negative integers 
$D(1,3m)$ is a generalization of Catalan numbers. More precisely, 
the number $D(1,3m)$ is the number of ways
for the permutation of $2m$ copies of $-1$ and $m$ copies of 2
such that the partial sum is never more than 0.
It is the sequence A001764 in OEIS and was named
as {\it $3$-Raney sequences} in Section 7.5 in the book `Concrete Mathematics' \cite{GKP}.
Similarly, the number $D(1,3m+1)$ is the number of ways
for the permutation of $2m+1$ copies of $-1$ 
and $m$ copies of 2 such that the partial 
sum is never more than 0. It is the sequence A006013 in OEIS.
\end{proof}

\noindent We have the following relation \[r(n,k)= \dfrac{D(n,k-1)}{2^k}\]because the player has $n-1$ or $n-2$ chips on his $(k-1)^{th}$ move, and 
on his next move he must collect $2$ chips to win the game. Also, 
it is intuitively clear that the number $q(n,k)$ in this case is less than or 
equal to $q(n,k)$ for the case $\{a,b\}=\{-1,1\}$ for any fixed $n$ and $k$. 
Therefore, $\displaystyle \lim_{k \to \infty}q(n,k) =0$  for each $n \geq 0$. 
For each $n\geq 1$, we evaluate 
$\displaystyle \sum_{k=1}^\infty r^2(n,k)$ 
and apply Theorem \ref{theorem1} to find the winning probability of the second player. \\

\noindent However, this sum is not Gosper-summable
(no {\it nice} closed-form formula for the partial sum). 
And these sums do not seem to converge to any famous constant
neither. We list their numerical values below.

 \begin{center}
    \label{tab:table1}
    \begin{tabular}{c|c|c} 
      $n$ & $\displaystyle \sum_{k=1}^\infty r^2(n,k)$ & $ p_n$ \\
      \hline
      1 & 0.3221721826105... &     0.33891390869471156...\\
      2 & 0.2886887304423... &     0.35565563477884626...\\
      3 & 0.1547549217692... &     0.42262253911538507...\\
      4 & 0.1241072133089... &     0.43794639334553199...\\
      5 & 0.0941564190484... &     0.45292179047578731...\\
      10 & 0.047917368748... &   0.47604131562562199...\\
      20 & 0.028469734522... &   0.48576513273891113...\\
      100 & 0.010952807500... &  0.49452359624969611...
    \end{tabular}
  \end{center}

\section{A remark on Theorem 1} \label{section4}
We would like to find an analog of Theorem \ref{theorem1} when
players have the same set of moves (add $a$ or $b$ chips),
but now the first player wins if he collects $n_1$ chips first 
and the second player wins if he collects $n_2$ chips first. \\ \\ 
For a fixed set of moves $\{a, b\}$, let 
$p_{n_1,n_2}$ be the probability that the second player 
collects $n_2$ chips before the first player collects $n_1$ chips. Then
\begin{equation} \label{new}
  p_{n_1,n_2} = \sum_{k=1}^{\infty} q(n_1,k)\cdot r(n_2,k). 
\end{equation}
Similarly, the probability that the first player collects
$n_1$ chips before the second player collects $n_2$ chips is 
\[  \sum_{k=1}^{\infty} q(n_2,k-1)\cdot r(n_1,k). \]

\noindent Assuming that the game will not go on indefinitely 
(the probability that either one of the players wins the game is 1), i.e., $a+b \geq 0$, we have 
\[  \sum_{k=1}^{\infty} q(n_2,k-1)\cdot r(n_1,k)
+ \sum_{k=1}^{\infty} q(n_1,k)\cdot r(n_2,k) = 1.\]

\noindent By Claim 1, we have the generalized version of Theorem \ref{theorem1} as follows:

\begin{prop} \label{proposition7}
For a fixed set of moves \{$a$,$b$\}, 
let  $p_{n_1,n_2}$ be the probability that the second player collects $n_2$ chips before the first player collects $n_1$ chips.
If the probability that either one of the player wins the game is 1, then
\[  p_{n_1,n_2} +p_{n_2,n_1} + \sum_{k=1}^{\infty} r(n_1,k)\cdot r(n_2,k) = 1.  \]
\end{prop}

\subsection{ $p_{n_1,n_2}$ for the case $ \{a,b\} = \{-1, 1\}$} \label{section4.1}
We list some values of $p_{n_1,n_2}$ obtained by the equation \eqref{new}.
Note again that, for any fixed $n$, the probability $q(n,k)$ has a nice closed form in $k$.
For example, previously, in Section 2.3, we obtained
\[ q(1,2m)=q(1,2m-1)= \dfrac{\binom{2m}{m}}{4^m}. \]
All of the infinite sums for different values of $n_1$ and $n_2$ in the equation \eqref{new} also simplify nicely.
For more values of $p_{n_1, n_2}$, one can use the function \texttt{Win2($n_1$, $n_2$)} in the 
accompanied program.

\newpage
\begin{table}[ht]

\centering 
\begin{tabular}{c | c c c c c} 
 \hline 
 $n_1$ $\backslash$ $n_2$ & 1 & 2 & 3 & 4 & 5\\ [0.5ex] \hline \\
1 & $\dfrac{\pi-2}{\pi}$ &  $\dfrac{4-\pi}{\pi}$ & $\dfrac{10-3\pi}{\pi}$ &
$\dfrac{3\pi-8}{3\pi}$ & $\dfrac{15\pi-46}{3\pi}$  \\ [3ex]
2 & $\dfrac{2(\pi-2)}{\pi}$ &  $\dfrac{3\pi-8}{\pi}$ & $\dfrac{2(10-3\pi)}{\pi}$ &
$\dfrac{3(16-5\pi)}{\pi}$ & $\dfrac{2(15\pi-46)}{3\pi}$  \\ [3ex]
3 & $\dfrac{3\pi-2}{3\pi}$ &  $\dfrac{7\pi-20}{\pi}$ & $\dfrac{39\pi-118}{3\pi}$ &
$\dfrac{296-93\pi}{3\pi}$ & $\dfrac{5(142-45\pi)}{3\pi}$  \\ [3ex]
4 & $\dfrac{8}{3\pi}$ &  $\dfrac{16-3\pi}{3\pi}$ & $\dfrac{8(12\pi-37)}{3\pi}$ &
$\dfrac{195\pi-608}{3\pi}$ & $\dfrac{8(63-20\pi)}{\pi}$  \\ [3ex]
5 & $\dfrac{5\pi-2}{5\pi}$ &  $\dfrac{92-27\pi}{3\pi}$ & $\dfrac{926-285\pi}{15\pi}$ &
$\dfrac{7(23\pi-72)}{\pi}$ & $\dfrac{5115\pi-16046}{15\pi}$  \\ [1ex]
\end{tabular}
\caption{The winning probability of the 
second player, $p_{n_1,n_2}$} 
\end{table}


\end{document}